\documentclass[11pt,a4paper]{amsart}
\usepackage{amscd,amssymb,amsopn,amsmath,amsthm,graphics,amsfonts,enumerate,verbatim,calc}
\usepackage[dvips]{graphicx}
\input xy
\xyoption{all}

\usepackage{amssymb,amsmath}

\headsep=1cm \numberwithin{equation}{section}
\newtheorem{theorem}{Theorem}[section]
\newtheorem{lemma}[theorem]{Lemma}
\newtheorem{proposition}[theorem]{Proposition}
\newtheorem{corollary}[theorem]{Corollary}

\theoremstyle{definition}
\newtheorem{definition}[theorem]{Definition}
\theoremstyle{remark}
\newtheorem{remark}[theorem]{Remark}
\newtheorem{example}[theorem]{Example}

\newtheorem{acknowledgement}{Acknowledgement}

\newcommand{\im}{\operatorname{im}}

\newcommand{\pgrade}{\operatorname{p.grade}}
\newcommand{\Kgrade}{\operatorname{K.grade}}

\newcommand{\cgrade}{\operatorname{c.grade}}

\newcommand{\Spec}{\operatorname{Spec}}

\newcommand{\Ht}{\operatorname{ht}}

\newcommand{\V}{\operatorname{V}}

\newcommand{\Supp}{\operatorname{Supp}}
\newcommand{\Tor}{\operatorname{Tor}}

\newcommand{\fm}{\frak{m}}
\newcommand{\fp}{\frak{p}}

\newcommand{\fa}{\frak{a}}
\newcommand{\fb}{\frak{b}}

\newcommand{\fn}{\frak{n}}

\newcommand{\ZZ}{{\mathbb Z}}
\newcommand{\RR}{{\mathbb R}}

\begin{document}
\author[Asgharzadeh]{Mohsen Asgharzadeh }
\author[Bhattacharyya]{Rajsekhar Bhattacharyya $^\ast$}

\title[Applications of closure operations on big Cohen-Macaulay Algebras]
{Applications of closure operations on big Cohen-Macaulay Algebras}

\address{School of Mathematics, Institute for Research in Fundamental Sciences (IPM), P.O. Box 19395-5746, Tehran, Iran.}
\email{asgharzadeh@ipm.ir}
\address{Dinabandhu Andrews College, Kolkata 700084, India}
\email{rbhattacharyya@gmail.com}

\keywords{Cohen-Macaulayness, Closure operation, Non-Noetherian rings, Perfect algebras.\\$^\ast$ Corresponding author}

\begin{abstract}
In this note we present some remarks on big Cohen-Macaulay algebras and almost Cohen-Macaulay algebras via closure operations on the ideals and the modules. Our methods for doing these are inspired by the notion of the tight closure and the dagger closure and by the ideas of Northcott on dropping of the Noetherian assumption of certain homological properties.
\end{abstract}

\maketitle

\section{Introduction}

Throughout this paper all rings are commutative, associative, with
identity, and all modules are unital. We always denote a commutative Noetherian ring by $R$ and an $R$-algebra (not necessarily Noetherian) by $A$.

The \emph{perfect closure} of a reduced ring $A$ of prime characteristic $p$ is defined by adjoining all the higher $p$-power roots of all the elements of $A$, to $A$ itself. We denote it by $A^\infty$. Observe the definition of coherent ring and non-Noetherian regular ring in Section 2. In Section 2, we show that the Frobenius map is flat over a coherent regular ring which is of prime characteristic.

In view of \cite[page 36]{E}, one may ask for a generalization of Tight Closure via \emph{Colon-Capturing} with respect to some non-Noetherian Cohen-Macaulay rings. Let $R$ be a  Noetherian local F-coherent domain which is either excellent or  homomorphic image of a Gorenstein local ring. If $\pgrade_A(\fa,M)$ is the  \emph{polynomial grade} of an ideal $\fa$ on an $A$-module $M$ (see Section 2 for the definition) then by applying a version of colon-capturing for $R^{\infty}$, in Theorem \ref{main} of Section 3, we show that $\Ht(\fa)=\pgrade_{R^{\infty}}(\fa, R^{\infty})$ for every ideal $\fa$ of $R^{\infty}$. On regular rings, this result follows easily by using a result of Kunz (which is true only for regular rings), see e.g. \cite[Theorem 4.5]{AT}. Following \cite{Sh}, a ring $R$ is called \emph{F-coherent} if $R^{\infty}$ is a coherent ring. In particular, Theorem \ref{main} provides an evidence  for Shimomoto's question: Does the perfect closure of an F-coherent domain coincide with the perfect closure of some regular Noetherian rings?

The \emph{Dagger closure} was introduced by Hochster and Huneke in \cite{HH}. It has an interesting connection with the theory of vector bundles, see e.g. \cite{BS} and the theory of \emph{tight closure}, see e.g. \cite{HH} and \cite{HH2}.  Let $I$ be an ideal of  $R$ and let $R^{+}$ be its integral closure in an algebraic closure of its fraction field. We recall  that an element $x\in R$ is in $I^{\dag}$, the dagger closure of $I$, if there are elements $\epsilon_n\in R^{+}$ of arbitrarily small order such that $\epsilon_nx\in IR^{+}$. By the main result of \cite{HH}, the tight closure of an ideal coincides with the dagger closure, where $R$ is complete and of prime characteristic $p$. In Section 4, we extend that notion to the submodules of finitely generated modules over $R$ to prove that if a complete local domain is contained in an almost Cohen-Macaulay domain then there exists a balanced big Cohen-Macaulay module over it (see Corollary \ref{main2}). In \cite{Ho}, Hochster showed the existence of a big Cohen-Macaulay algebra from the existence of an almost Cohen-Macaulay algebra in dimension 3. Corollary \ref{main2} extends this result, by proving the existence of big Cohen-Macaulay modules from the existence of almost Cohen-Macaulay algebras in any dimension $d$.

In \cite{S1}, certain weakly almost Cohen-Macaulay $R^+$-algebra is constructed and it has been asked whether it is almost Cohen-Macaulay or not. In Section 5, we obtain that such a weakly almost Cohen-Macaulay $R^+$-algebra will be almost Cohen-Macaulay if and only if it maps to some big Cohen-Macaulay $R^+$-algebra. For more details see Remark 5.4.

\section{Preliminary Notations}

We begin this section by exploring the following definitions.

Let $A$ be an algebra  equipped with a map $v:A\to\RR \cup \{\infty\}$ satisfying $(1)$
$v(ab)=v(a)+v(b)$ for all $a,b \in A$;
$(2)$
$v(a+b) \ge \min\{v(a),v(b)\}$ for all $a,b \in A$ and
$(3)$
$v(a)=\infty$ if and only if $a=0$. The map $v$ is called the \emph{value map} on $A$.
Moreover, if $v(a) \ge 0$ for every $a \in T$ and $v(a) > 0$ for every non-unit $a \in A$, then we say that $v$ is \emph{normalized}. If  $A$ is equipped with a normalized value map, then by \cite[Proposition 3.1]{AS}, one has $\bigcap_{n=1}^{\infty} \fa^n=0$  where $\fa$ is proper and finitely generated ideal of $A$.

Let $A$ be of prime characteristic $p$ and $I$ be an  ideal of $A$. Set $A^0$ to be the  complement of the set of all minimal primes of  $A$ and $I^{[q]}$ to be the ideal generated by $q=p^e$-th powers of all elements of $I$. Then the tight closure $I^\ast$ of $I$ is the set of $x\in A$ such that there exists $c \in A^0$  such that $cx^q\in I^{[q]}$ for $q\gg 0$.

We also recall the following definitions: a ring is called coherent if each of its finitely generated ideal is finitely presented and a ring is called \emph{regular} if each of its finitely generated ideal is of finite projective dimension.

\begin{lemma} \label{sh1}
Let $A$ be a  coherent regular ring of prime characteristic $p$.
Then the following statements hold. \begin{enumerate}
\item[$\mathrm{(i)}$]  The Frobenius map is flat.
\item[$\mathrm{(ii)}$]  If $A$ is equipped with a normalized value map, then all finitely generated ideals of $A$ are tightly closed.
\end{enumerate}
\end{lemma}

\begin{proof}

$\mathrm{(i)}$: This proof is inspired by a famous result of Kunz \cite[Corollary 8.2.8]{BH},
which characterizes Noetherian regular rings of prime characteristic in the term of Frobenius map.
  First recall that the assignment $a\mapsto a^p$ defines the Frobenius map $F:A\to A$ which is a ring homomorphism.  By $F(A)$, we mean $A$ as a group equipped with  left and right scalar multiplication from $A$
given by
$$a.r\star b = ab^pr, \  \ where \ \ a,b\in A  \ \ and \  \ r\in F(A),$$
see also \cite[Section 8.2]{BH}. We show that for all $i > 0$, $\Tor^A_i(A/\fa, F(A))=0$ for all finitely generated ideals $\fa\subset A$. Note that $A/\fa$  has a  free resolution $(F_\bullet,d_\bullet)$ consisting of finitely generated modules, since $A$ is coherent. Such a resolution is bounded, because $A$ is regular. Then $(F_\bullet,d_\bullet)\otimes_A F(A)=(F_\bullet,d_\bullet^p)$.

Consider the ideal $I_{t}((a_{ij}))$ generated by the $t\times t$ minors of a matrix $(a_{ij})$ and recall the definition of Koszul grade, which is denoted by $\Kgrade$. One can show that it is unique up to the radical, see \cite[Proposition 2.2 (vi)]{BH}. Let $r_i$ be the expected rank of $d_\bullet$, see \cite[Section 9.1]{BH}. Clearly, $r_i$ is the expected rank of $d_\bullet^p$. Now, all these facts together imply, $$\Kgrade_A(I_{r_i}(d_i), A)=\Kgrade_A(I_{r_i}(d_i^p ),A).$$
In view of \cite[Theorem 9.1.6]{BH}, which is a beautiful theorem of Buchsbaum-Eisenbud-Northcott, we find that $(F_\bullet,d_\bullet^p)$ is exact and so $\Tor^A_i(A/\fa, F(A))=0$.

$\mathrm{(ii)}$: This proof is inspired by \cite[Lemma 4.1]{Sh}. Let $\fa$ be a proper and finitely generated ideal of $A$. Suppose that $\fa^\ast\neq\fa$. Take $x\in \fa^\ast\setminus \fa$. From the definition and from part $(i)$ of the proof, it follows that there exists an element $c \in A^0$  such that $c\in( \fa^{[q]}:_Ax^q)=(\fa:_Ax)^{[q]}$ for $q\gg 0$. We recall that $\bigcap_{n=1}^{\infty} \fb^n=0$ for any proper and finitely generated ideal of $A$, since $A$ is equipped with a normalized value map. Here $(\fa:_Ax)$ is proper, since $x\notin \fa$ and it is also finitely generated, since $A$ is coherent. These facts together imply $c\in\bigcap (\fa:_Ax)^{[q]}\subseteq\bigcap (\fa:_Ax)^q=0$ and thus we arrive at a contradiction.
\end{proof}

It is to be noted that the finitely generated assumption of the previous result is really
needed.
\begin{example}\label{exam}
Let $(R,\fm)$ be a Noetherian regular local ring of prime characteristic $p$ which is not a field. Let $R^\infty$ be its  perfect closure. In view of
 \cite[Theorem 1.2]{A}, $R^\infty$ is coherent and of finite global dimension. This is well-known that, $R^\infty$ is equipped with a normalized value map. To construct such a value map, consider a discreet valuation valuation ring $(V,\fn)$, which birationally dominates $R$, i.e., $R\subseteq V$, $\fn\cap R=\fm$ and both $R$ and $V$ have the same field of fractions. We know that the above situation exists. Let $v$ be a value map of $V$ and take $y\in R$ such that $v(y)=\ell\in \mathbb{N}\setminus\{0\}$. Let $r\in R^\infty$. Then $r^{p^n}\in R$ for some $n$. The assignment $r\mapsto v(r^{p^n})/p^n$ defines a normalized value map on $R^\infty$.
Consider the ideal $\fa:=\{x\in R^\infty:v(x)>\ell/p\}$. Here, we show that $\fa^\ast\neq \fa$. To see this, let $x\in R^\infty$ be the p-th root of $y\in R$. Here $v(x)=\ell/p$.  Take $c\in R^\infty$ with $v(c)>0$. Clearly,  $v(c^{1/q}x)>\ell/p$ and so  $c^{1/q}x\in \fa.$  Thus $cx^q=\prod_q c^{1/q}x\in\fa^{[q]}$ for $q\gg 0$. Therefore, $x\in \fa^\ast\setminus \fa$. By Lemma~\ref{sh1}, $\fa$ can not be finitely generated. Here, we show (directly) that $\fa$ is not finitely generated. The proof of this fact goes along the same line of that of \cite[Proposition 6.8]{A}, which gives a similar result over the absolute integral closure of $R$. We leave the details to the reader.
\end{example}

\section{Cohen-Macaulayness of Minimal Perfect Algebras}

We begin this section with the following lemma.

\begin{lemma}  \label{sh}
Let $(R,\fm)$ be a  Noetherian local F-coherent domain which is either excellent or homomorphic image of a Gorenstein local ring. Then $R^{\infty}$ is balanced  big Cohen-Macaulay.
\end{lemma}

\begin{proof}
This is proved in \cite[Theorem 3.10]{Sh} when  $R$ is a homomorphic image of a Gorenstein ring. The argument of \cite[Theorem 3.10]{Sh} based on \emph{Almost Ring Theory}. Assume that $R$ is excellent. For each n, set $R_n:=\{x\in R^\infty|x^{p^n}\in R\}$. The assignment $x\mapsto x^{p^n}$ shows that $R\simeq R_n$, which implies that $R_n$ is excellent. We recall that over excellent domains of prime characteristic one can use the colon capturing property of tight closure theory, see e.g. \cite[Theorem 1.7.4]{HH2}. Let $\underline{x}:=x_1,\ldots,x_d$ be a system of parameters for $R$, where $d:=\dim R$ and let $r\in R^{\infty}$ be such that $rx_{i+1}=\sum_{j=1}^i r_jx_j$ for some $r_j\in R^{\infty}$. Then $r,r_j\in R_n$ for $n\gg 0$. So $r\in((x_1,\ldots, x_i)R_n :_{R_n} x_{i+1})$. Thus from Lemma \ref{sh1},
\[\begin{array}{ll}
((x_1,\ldots, x_i)R^{\infty} :_{R^{\infty}} x_{i+1})&=\bigcup_n((x_1,\ldots, x_i)R_n :_{R_n} x_{i+1})\\
&\subseteq \bigcup_n((x_1,\ldots, x_i)R_n)^\ast\\
&\subseteq ((x_1,\ldots, x_{i})R^\infty)^\ast\\
&=(x_1,\ldots, x_{i})R^\infty.
\end{array}\]
Clearly, $\fm R^{\infty}\neq R^{\infty}$. So, any system of parameters of $R$ is regular sequence on $R^{\infty}$, i.e., $R^{\infty}$ is balanced  big Cohen-Macaulay.
\end{proof}

It is to be noted that in the case of non-Noetherian rings, tight closure may behave nicely, since from the above proof we find that it exhibits the colon capturing property.

We recall the following definitions. Let $\fa$ be an ideal of a ring $A$ and $M$ be an $A$-module.  A finite sequence $\underline{x}:=x_{1},\ldots,x_{r}$ of elements of $A$ is called  $M$-sequence if $x_i$ is a nonzero-divisor on $M/(x_1,\ldots, x_{i-1})M$ for $i=1,\ldots,r$ and $M/\underline{x}M\neq0$. The classical grade  of $\fa$ on $M$, denoted by $\cgrade_A(\fa,M)$, is defined by the supremum length of maximal $M$-sequences in $\fa$. The polynomial grade of $\fa$ on $M$ is defined by $$\pgrade_A(\fa,M):=\underset{m\rightarrow\infty}{\lim}\cgrade_{A[t_1, \ldots,t_m]}(\fa A[t_1, \ldots,t_m],A[t_1,\ldots,t_m]\otimes_A M)$$

and we will use the following well-known properties of the polynomial grade.

\begin{lemma}\label{pro} (see e.g. \cite{AT}) Let $\fa$ be an ideal of a ring $A$ and $M$ an
$A$-module. Then the following statements hold.
\begin{enumerate}
\item[$\mathrm{(i)}$] If $\fa$ is finitely generated, then$$\pgrade_A(\fa,M)=\inf\{\pgrade_{A_{\fp}}(\fp A_{\fp},M_{\fp})|\fp\in \V(\fa)\cap\Supp_A M\}.$$
\item[$\mathrm{(ii)}$] Let $\Sigma$ be the family of all finitely generated  ideals $\fb\subseteq\fa$. Then$$\pgrade_A(\fa,M)=\sup\{\pgrade_A(\fb,M):\fb\in\Sigma\}.$$
\item[$\mathrm{(iii)}$] $\pgrade_A(\fa,M)\leq\Ht_{M}(\fa).$
\end{enumerate}
\end{lemma}

Now, we are ready to prove the following:

\begin{theorem}\label{main}
Let $R$ be a  Noetherian local F-coherent domain which is either excellent or  homomorphic image of a Gorenstein local ring. Then  $\Ht(\fa)=\pgrade_{R^{\infty}}(\fa, R^{\infty})$ for all  ideal $\fa$ of $R^{\infty}$.
\end{theorem}

\begin{proof}  Let $\underline{x}:=x_1,\ldots,x_d$ be a system of parameters for $R$, where $d:=\dim R$. By Lemma \ref{sh}, $\underline{x}$ is regular sequence on $R^{\infty}$. So $$d\leq\pgrade(\underline{x}R^{\infty},R^{\infty})\leq\pgrade(\fm_{R^{\infty}},R^{\infty})\leq\Ht(\fm_{R^{\infty}})=d.$$This shows that $\pgrade(\fm_{R^{\infty}},R^{\infty})=\Ht(\fm_{R^{\infty}})$  for the maximal ideal $\fm_{R^{\infty}}$ of $R^{\infty}$.
At first, we assume that $\fa$ is  finitely generated and let  $P$ be a prime ideal of $R^{\infty}$ such that $P\supseteq \fa$. Set $\fp:=P\cap R$.
Take $x\in(R_{\fp})^{\infty}$. Then $x^{p^n}\in R_{\fp}$ for some $n$, where $p$ is the characteristic of $R$.  Thus $x^{p^n}=a/b$ for some $a\in R$ and $b \in R\setminus\fp$. Consider $\frac{a^{1/p^n}}{b^{1/p^n}}$ as an element of $R^{\infty}_P$. Let $m\geq n$. Then  $\frac{a^{1/p^n}}{b^{1/p^n}}=(\frac{a^{1/p^{m-n}}}{b^{1/p^{m-n}}})^{1/p^m}$. Keep it in the mind that $R$ is of prime characteristic. If we assume that $a/b=c/d$ then $\frac{a^{1/p^n}}{b^{1/p^n}}=\frac{c^{1/p^n}}{d^{1/p^n}}$. Putting these facts together, we see that the assignment $x\mapsto \frac{a^{1/p^n}}{b^{1/p^n}}$ is independent of the presentation of $x$ and of the choice of $n$. So, it defines a ring homomorphism  $\varphi: (R_{\fp})^{\infty} \to R^{\infty}_P$. Clearly, $\ker \varphi =0$. Let $x\in R^{\infty}_P$. Then $x=\alpha /\beta$, where $\alpha\in R^{\infty}$ and $\beta\in R^{\infty}\setminus P$. Take $n$ to be such that
$a:=\alpha^{p^n}\in R$ and $b:=\beta^{p^n}\in R\setminus \fp$. Set $c:=a/b\in R_{\fp}$.  Consider $\gamma:=c^{1/p^n}\in (R_{\fp})^{\infty}$. Then  $\varphi(\gamma)=x$, and thus $\varphi$ is an isomorphism.

We recall that $R_{\fp}$ is F-coherent.  Clearly,  $R_{\fp}$ is either excellent or  homomorphic image of a Gorenstein local ring. By the case of maximal ideals,  $$\pgrade(PR^{\infty}_P,R^{\infty}_P)=\pgrade(\fm_{(R_{\fp})^{\infty}},(R_{\fp})^{\infty})=\Ht(\fp)=\Ht(P).$$ This fact along with Lemma \ref{pro} yields that
\[\begin{array}{ll}
\pgrade(\fa,R^{\infty})&=\inf\{\pgrade(PR^{\infty}_P,R^{\infty}_P)|P\in \V(\fa)\}\\
&=\inf\{\Ht(P)|P\in \V(\fa)\}\\&=\Ht(\fa),
\end{array}\]
which  shows that $\pgrade(\fa,R^{\infty})=\Ht(\fa)$ for every finitely generated ideal $\fa$ of $R^{\infty}$.

Finally we assume that $\fa$ is an arbitrary ideal of $R^{\infty}$ and let $\Sigma$ be the family of all finitely generated ideals $\fb\subseteq\fa$. We  bring the following claim:

 Claim: One has $\Ht(\fa)=\sup\{\Ht(\fb):\fb\in \Sigma\}$.

To see this, let $P\in \Spec(R^{\infty})$ be such that $\Ht(P)=\Ht(\fa)$ and set $\fp:=P\cap R$ and we already have that  $R^{\infty}_P=(R_{\fp})^{\infty}$. Set $n:=\Ht(\fp)=\Ht(P)=\Ht(\fa \cap R)$. Due to \cite[Theorem A.2]{BH}, there exists a sequence $\underline{x}:=x_1,\ldots,x_n$ of elements of $\fa\cap R$ such that $\Ht(x_1,\ldots,x_i)R=i$ for all $1\leq i \leq n$.  Since $R$ is catenary, $\underline{x}$ is part of a system of parameters for $R$. By Lemma \ref{sh}, $\underline{x}$ is a regular sequence on $R^{\infty}$. So $$\Ht(P)\geq \Ht(\underline{x}R^{\infty})\geq\pgrade(\underline{x}R^{\infty},R^{\infty})=n= \Ht(\fp)=\Ht(P),$$
which shows that $\Ht(\underline{x}R^{\infty})=\Ht(\fa)$. This completes the proof of the claim.

From the results of  finitely generated ideals, we have that $\pgrade(\fb,R^{\infty})=\Ht(\fb)$ for all $\fb\in\Sigma$. Thus from Lemma \ref{pro} and from the above claim we see that
\[\begin{array}{ll}
\pgrade(\fa,R^{\infty})&=\sup\{\pgrade(\fb,R^{\infty}):\fb\in\Sigma\}\\
&=\sup\{\Ht(\fb):\fb\in \Sigma\}\\&=\Ht(\fa),\\
\end{array}\]
and this is precisely what we wish to prove.
\end{proof}

Let $M$ be an $A$-module. Recall that a prime ideal $\fp$ is weakly associated to $M$ if $\fp$ is minimal over $(0 :_{A} m)$ for some $m\in M $.

\begin{corollary}
Adopt the assumption of Theorem \ref{main} and let $\fa$ be an ideal of $R^{\infty}$ which is generated by
$\Ht(\fa)$ elements. Then all the weakly associated prime ideals of $R^{\infty}/ \fa$ have the same height.
\end{corollary}

\begin{proof}
See \cite[Corollary 4.6]{AT} and its proof.
\end{proof}

\begin{remark}\label{rem}
$\mathrm{(i)}$: Concerning the proof of Theorem \ref{main}, we can ask the following question which has a positive answer for several classes of commutative rings such as valuation domains. Let $A$ be a commutative ring and $\fa$  an  ideal of $A$. Let $\Sigma$ be the family of all finitely generated ideals $\fb\subseteq\fa$. Is $\Ht(\fa)=\sup\{\Ht(\fb):\fb\in \Sigma\}$?

$\mathrm{(ii)}$: Let A be a ring with the property that $\pgrade_A(\fm;A) = \Ht(\fm)$ for all maximal
ideals $\fm$ of $A$. One may ask whether $\pgrade_A(\fa;A) = \Ht(\fa)$ holds for all ideals
$\fa$ of $A$. In view of \cite[Example 3.11]{AT}, this is not true in general.
\end{remark}

\section{Dagger Closure and Big Cohen-Macaulay algebras}

 In this section we extend the notion of dagger closure to the submodules of a finitely generated module over a Noetherian local domain $(R,\fm)$ and we present Corollary \ref{main2}.

\begin{definition}\label{def:almost closure of an ideal}
Let $A$ be a local algebra with a normalized valuation $v:A\rightarrow \RR\cup\{\infty\}$ and $M$ be an $A$-module. Consider a submodule $N\subset M$. We say $x\in N^{v}_M$ if for every $\epsilon> 0$, there exists $a\in A$ such that $v(a)< \epsilon$ and $ax\in N$.
\end{definition}

\begin{definition}\label{almcm}
Let $A$ be an algebra over a Noetherian local domain $(R,\fm)$ of $\dim d$. Assume that $A$ is equipped  with a normalized value map $v : A\to \RR\cup\{\infty\}$. From \cite{RSS}, we recall that $A$ is called almost Cohen-Macaulay if for $i=1,\ldots, d$, each element of $((x_{1},\ldots, x_{i-1})A:_{A}x_{i})/(x_{1},\ldots,x_{i-1})A$ is annihilated by elements of sufficiently small order with respect to $v$ for all system of parameters $x_1,\ldots, x_d$ of $A$.
\end{definition}

\begin{definition}\label{def}
Let $(R,\fm)$ be a Noetherian local domain and let $A$ be a local domain containing $R$ with a normalized valuation $v:A\rightarrow \RR\cup\{\infty\}$. For any finitely generated $R$-module $M$ and for its submodule $N$ we define submodule $N_{M}^{\bold{v}}$ such that $x\in N_{M}^{\bold{v}}$ if $x\otimes 1 \in \im(N\otimes A\to M\otimes A)_{M\otimes A}^{v}$.
\end{definition}

\begin{remark} Let $A$ be a perfect domain and $\frak a$ be a nonzero radical ideal of $A$. Then  $\frak a^v=A$. To see this, let $x$ be a  nonzero element of $\frak a$. Since $v(x^{1/n})=v(x)/n$ and $x^{1/n}\in \frak a$, the conclusion follows.
\end{remark}

\begin{proposition}\label{pr}
Let $(R,\fm)$ be a Noetherian local domain and let $A$ be a local domain containing $R$ with a normalized valuation $v:A\rightarrow \RR\cup\{\infty\}$. Let $M$, $M'$ be finitely generated $R$-modules. Consider the submodules $N$, $W$ of $M$. Then the following statements are true:\begin{enumerate}
\item[$\mathrm{(i)}$]  $N^{\bold{v}}_M$ is a submodule of $M$ containing $N$.
\item[$\mathrm{(ii)}$]$(N^{\bold{v}}_M)^{\bold{v}}_M= N^{\bold{v}}_M$.
\item[$\mathrm{(iii)}$]If $N\subset W\subset M$, then $N^{\bold{v}}_M\subset W^{\bold{v}}_M$.
\item[$\mathrm{(iv)}$] Let $f:M\to M'$ be a homomorphism. Then $f(N^{\bold{v}}_M)\subset f(N)^{\bold{v}}_{M'}$.
\item[$\mathrm{(v)}$]
If $N^{\bold{v}}_M= N$, then $0^{\bold{v}}_{M/N}= 0$.
\item[$\mathrm{(vi)}$]
We have $0^\bold{v}_R=0$ and $\fm^\bold{v}_R=\fm$.\end{enumerate}
In addition to, if $A$ is almost Cohen-Macaulay, then the following is true:
\begin{enumerate}
\item[$\mathrm{(vii)}$] Let $x_1, \ldots , x_{k+1}$ be a partial system of parameters for $R$, and let $J =
(x_1, \ldots , x_k)R$. Suppose that there exists a surjective homomorphism $f:M\to
R/J$ such that $f(u) = \bar{x}_{k+1}$, where $\bar{x}$ is the image of $x$ in $R/J$. Then
$(Ru)^{\bold{v}}_M\cap \ker f \subset (Ju)^{\bold{v}}_M$.
\end{enumerate}
\end{proposition}

\begin{proof}
The proof is straightforward and we leave it to the reader.
\end{proof}

In \cite{Ho}, Hochster showed the existence of a big Cohen-Macaulay algebra from the existence of an almost Cohen-Macaulay algebra in dimension 3. The following Corollary extends this result, by proving the existence of big Cohen-Macaulay modules from the existence of almost Cohen-Macaulay algebras in any dimension $d$.

\begin{corollary}\label{main2}
For a complete Noetherian local domain, if it is contained in an almost Cohen-Macaulay domain, then there exists a balanced big Cohen-Macaulay module over it.
\end{corollary}

\begin{proof}
Due to \cite[Theorem 3.16]{D}, we know that there exists a list of seven axioms for a closure operation defined
for finitely generated modules over  complete local domains. By the main result of  \cite{D}, one can  see
that a closure operation satisfying those axioms implies the existence of a balanced
big Cohen-Macaulay module over the base ring. In view of Proposition \ref{pr}, we see that
the closure operation defined in Definition \ref{def} satisfies
all the seven axioms of \cite[Theorem 3.16]{D}. This completes the proof.
\end{proof}

The above result can be used to prove \cite[Proposition 1.3]{RSS} by a completely different argument, since the existence of almost Cohen-Macaulay domain impiles the existence of big Cohen-Macaulay module.

\section{Concluding Remarks}

For Noetherian local domain $R$ of $\dim d$, the above results can be extended to an almost Cohen-Macaulay $R^{+}$-algebra where $R^{+}$ be its integral closure in an algebraic closure of its fraction field. Fixed a local $R^{+}$-algebra $A$ and let $M$ be an $A$-module. Consider a submodule $N\subset M$. In view of Definition \ref{def:almost closure of an ideal}, define $ N^{v}_M$ via a normalized valuation $v:R^{+}\rightarrow \RR\cup\{\infty\}$. For a finitely generated $R$-module $M$ and its submodule $N$ we define submodule $N_{M}^{\bold{v}}$ such that $x\in N_{M}^{\bold{v}}$ if $x\otimes 1 \in \im(N\otimes A\to M\otimes A)_{M\otimes A}^{v}$.

\begin{definition}
From \cite{S1}, we recall that $A$ is almost Cohen-Macaulay if for every system of parameters $x_1,\ldots, x_d$, $((x_1,\ldots, x_{i-1})A:_Ax_i)/(x_1,\ldots, x_{i-1})A$ is almost zero $R^{+}$-module for $i=1,\ldots,d$ and $A/\fm  A$ is not almost zero.
\end{definition}

\begin{corollary}
Let $(R,\fm)$ be a Noetherian local domain and let $R^{+}$ be equipped with a normalized valuation $v:R^{+}\rightarrow \RR\cup\{\infty\}$. Consider a local $R^{+}$-algebra $A$ and for every $I\subset R$, we define $I^{\bold{v}}$ as above. Under this situation, closure operation satisfies all the properties given in Proposition \ref{pr} if and only if $A$ is an almost Cohen-Macaulay $R^{+}$-algebra where $R^{+}$ is contained in it as a subdomain.
\end{corollary}

\begin{proof}
We show $0^{\bold{v}}= 0$ implies $R^{+}\to A$ is injective. To see this, take $0\neq a\in R^{+}$ such that its image in $A$ is zero. Since $a$ is integral over $R$, there exists a minimal monic expression $a^n+ r_1a^{n-1}+\cdots +r_n= 0$, where each $r_i\in R$ with $r_n\neq 0$. Take the image of the expression in $A$ which gives that the image of $r_n$ is zero in $A$. So $r_n\in 0^{\bold{v}}= 0$. This gives $r_n= 0$. So we arrive at a contradiction and $a$ is zero in $R^{+}$. By using straightforward modification of the proof of Corollary \ref{main2}, we get the rest of the claim.

The proof of the converse is left it to the reader.
\end{proof}

The following result provides an example of an almost Cohen-Macaulay $R^{+}$-algebra where $R^{+}$ is contained in it as a subdomain. We recall that an $R$-algebra $T$ is called a \emph{seed}, if it maps to a big Cohen-Macaulay $R$-algebra (see \cite{D1}).

\begin{proposition}\label{rem}
Let $(R,\fm)$ be a Noetherian complete local domain in mixed characteristic $p>0$, with a system of parameters $p, x_2,\ldots, x_d$. Let $B$ be as \cite[Theorem 5.3]{S1}. Then the following are equivalent.
\begin{enumerate}
\item[$\mathrm{(i)}$] $B/\fm B$ is not almost zero when viewed as $R^{+}$-module.
\item[$\mathrm{(ii)}$] $p^{\epsilon}\notin (p, x_2,\ldots, x_d)B$ for some rational number $\epsilon>0$.
\end{enumerate}
\end{proposition}

\begin{proof}
$\mathrm{(i)}\Rightarrow\mathrm{(ii)}$: Clearly $\fm B$ does not contain elements of $R^{+}$ of arbitrary small order. If $p^{\epsilon}\in (p, x_2,\ldots, x_d)B$ for every rational number $\epsilon$ then $p^{1/k}\in (p, x_2,\ldots, x_d)B$ for every positive integer $k$. Since $v(p)< \infty$, $p^{1/k}$ are the elements of $R^{+}$ of arbitrarily small order and $\fm B$ contains all of them. This is a contradiction.

$\mathrm{(ii)}\Rightarrow\mathrm{(i)}$:
From \cite[Corollary 6.1]{S1} it follows that $R^{+}$ is a seed and using similar argument of \cite[Example 5.2]{D1} we find that $R^{+}$ is actually a minimal seed. Thus $R^{+}$ can be thought as a subdomain of $B$. Let $\{x_{\lambda}\}= B-R^{+}$. Take $C= R^{+}[\{X_{\lambda}\}]$ where $X_{\lambda}$'s are the indeterminates. One can extend the valuation $R\to \ZZ \cup \{\infty\}$ to $C\to \RR \cup \{\infty\}$ such that $v$ is non negative on $C$ and positive on non-units of $C$. Also, we choose values of $X_{\lambda_i}$'s greater than some $N> 0$. For $B= C/JC$, we have $y\in JC$ implies $v(y)> N> 0$, and since $\fm$ is also finitely generated same is true for $\fm C$. Thus $C/(J+\fm)C$ is not almost zero viewing as $C$-module. Since $B/\fm B= C/(J+\fm)C$, we find that $B/\fm B$ is not almost zero viewing as $C$-module. Thus $B/\fm B$ is not almost zero viewing as $R^{+}$-module.
\end{proof}

\begin{remark}
In \cite{S1}, it has been asked that whether the $R^+$-algebra $B$ which satisfies conditions of Theorem 5.3 (there) is almost Cohen-Macaulay or not (i.e $B/mB$ is almost zero or not). From Corollary 5.2 and Proposition \ref{rem}, we get the following: Let $B$ be a weakly almost Cohen-Macaulay $R^+$-algebra satisfies Theorem 5.3 of \cite{S1}. Then $B$ is almost Cohen-Macaulay $R^+$-algebra (i.e. $B/mB$ is almost zero) if and only if it maps to big Cohen-Macaulay $R^+$-algebra. It is to be noted that here $R^+$ is contained in $B$ as a subdomain.
\end{remark}

\begin{acknowledgement}
We would like to thank K. Shimomoto for his comments on the earlier
version of this paper. We also thank the referee
for careful reading of the paper and for the valuable suggestions.
\end{acknowledgement}

\end{document}